\documentclass[reqno]{amsart}
\usepackage{amssymb,amsmath}
\usepackage{amsthm}
\usepackage{color,graphicx}
\usepackage{eucal}

\usepackage[pagewise,mathlines,displaymath]{lineno}

\newcommand{\R}{\mathbb R}

\newcommand{\p}{\partial}

\newcommand{\trinorm}{|\!|\!|}

\newcommand{\pa}{\partial}

\newtheorem{theorem}{Theorem}[section]
\newtheorem{proposition}[theorem]{Proposition}
\newtheorem{remark}[theorem]{Remark}
\newtheorem{lemma}[theorem]{Lemma}
\newtheorem{corollary}[theorem]{Corollary}

\begin{document}

\vglue-1cm \hskip1cm
\title[The generalized Zakharov-Kuznetsov equation]{A note on the 2D generalized Zakharov-Kuznetsov equation: Local, global, and
scattering results}

\author[L. G. Farah]{Luiz G. Farah}
\address{ICEx, Universidade Federal de Minas Gerais, Av. Ant\^onio Carlos, 6627, Caixa Postal 702, 30123-970,
Belo Horizonte-MG, Brazil}
\email{lgfarah@gmail.com}

\author[F. Linares]{Felipe Linares}
\address{IMPA, Estrada Dona Castorina 110, 22460-320, Rio de Janeiro-RJ,
 Brazil.}
\email{linares@impa.br}

\author[A. Pastor]{Ademir Pastor}
\address{IMECC-UNICAMP, Rua S\'ergio Buarque de Holanda, 651, 13083-859, Cam\-pi\-nas-SP, Bra\-zil}
\email{apastor@ime.unicamp.br}


\date{}



\begin{abstract}
We consider the generalized two-dimensional Zakharov-Kuznetsov
equation $u_t+\partial_x \Delta u+\partial_x(u^{k+1})=0$, where
$k\geq3$ is an integer number. For $k\geq8$ we prove local
well-posedness in the $L^2$-based Sobolev spaces
$H^s(\mathbb{R}^2)$, where $s$ is greater than the critical scaling
index $s_k=1-2/k$. For $k\geq 3$  we also establish a sharp criteria
to obtain global $H^1(\R^2)$ solutions. A nonlinear scattering
result in $H^1(\R^2)$ is also established assuming the initial data
is small and belongs to a suitable Lebesgue space.
\end{abstract}

\maketitle

\section{Introduction}\label{introduction}

This note sheds new light on the local and global well-posedness of
the initial-value problem (IVP) associated with the generalized
Zakharov-Kuznetsov (gZK) equation in two-space dimensions:
\begin{equation}\label{gzk}
\begin{cases}
{\displaystyle u_t+\partial_x \Delta u+ \partial_x(u^{k+1})  =  0,  }  \qquad (x,y) \in \mathbb{R}^2, \,\,\,\, t>0, \\
{\displaystyle  u(x,y,0)=u_0(x,y)},
\end{cases}
\end{equation}
where $u$ is a real-valued function,
$\Delta=\partial_x^2+\partial_y^2$ stands for the Laplacian
operator, and $k\geq1$ is an integer number. Here we will concern
with the $L^2$-supercritical case, i.e. $k\geq 3$ in \eqref{gzk}.

In the case $k=1$, the equation in \eqref{gzk} has a physical
meaning and it was formally deduced by Zakharov and Kuznetsov
\cite{ZK} as an asymptotic model to describe the propagation of
nonlinear ion-acoustic waves in a magnetized plasma. The gZK
equation  may also be seen as a natural, two-dimensional extension
of the well-known  generalized Korteweg-de Vries (KdV) equation
\[ u_t+u_{xxx}+\partial_x(u^{k+1})  =  0, \qquad
x\in\mathbb{R},\,\,\, t>0.
\]

Our main purpose here lies in establishing local and global (in
time) well-posedness results. These issues have already been studied
in Faminskii \cite{Fa}, Biagioni and Linares \cite{BL}, and Linares
and Pastor \cite{LP}, \cite{LP1}.  In \cite{Fa}, Faminskii
considered the the case $k=1$ and showed local and global
well-posedness  in $H^m(\mathbb{R}^2)$, $m\geq1$ integer. In
\cite{BL}, Biagioni and Linares dealt with the case $k=2$ and proved
local well-posedness for data in $H^1(\mathbb{R}^2)$. By considering
the cases $k=1$ and $k=2$ Linares and Pastor \cite{LP} improved the
local results in \cite{BL}, \cite{Fa} by showing that both IVP's are
locally well-posed  in $H^s(\mathbb{R}^2)$, $s>3/4$. Moreover the
authors also show that  if $u_0\in H^1(\mathbb{R}^2)$ and satisfies
$\|u_0\|_{L^2}< \|Q\|_{L^2}$, where $Q$ is the unique positive
radial solution (hereafter refereed to as the ground state solution)
of the elliptic equation
\begin{equation}\label{solwave}
-\Delta Q+Q-Q^3=0,
\end{equation}
then (for $k=2$) global well-posedness holds in $H^1(\mathbb{R}^2)$.
The case $k\geq3$ was studied in \cite{LP1} where the authors
established local well-posedness in $H^s(\mathbb{R}^2)$, $s>3/4$, if
$3\leq k \leq7$ and in $H^s(\mathbb{R}^2)$, $s>s_k^*:=1-3/(2k-4)$,
if $k\geq8$. A global result for small initial  data in
$H^1(\mathbb{R}^2)$ was also proved.

The best local well-posedness results known are summarized in the
following theorem.

\begin{theorem}[\cite{LP},\cite{LP1}]\label{reviewtheo}
The following statements holds.
\begin{itemize}
  \item[(i)] Assume $1\leq k\leq7$. Then for any $u_0 \in H^s(\mathbb{R}^2)$,
$s>3/4$, there exist $T=T(\|u_0\|_{H^s})>0$, a space $X_T\subset
C([0,T];H^s(\mathbb{R}^2))$ and a unique solution $u\in X_T$  of the
IVP \eqref{gzk} defined in $[0,T]$. Moreover, continuous dependence
upon the data holds.
  \item[(ii)] Assume $k\geq8$. Then for any $u_0 \in H^s(\mathbb{R}^2)$,
$s>s_k^*:=1-3/(2k-4)$, there exist $T=T(\|u_0\|_{H^s})>0$, a space
$Z_T\subset C([0,T];H^s(\mathbb{R}^2))$ and a unique solution $u\in
X_T$  of the IVP \eqref{gzk} defined in $[0,T]$. Moreover,
continuous dependence upon the data holds.
\end{itemize}
\end{theorem}

Concerning other questions on the gZK equation we refer the reader
to \cite{BIM}, \cite{dB}, \cite{LPS}, \cite{Pa}, \cite{PS}, and
references therein.

To motivate the results to follow, let us perform a scaling
argument: if $u$ solves \eqref{gzk}, with initial data $u_0$, then
$$
u_\lambda(x,y,t)=\lambda^{2/k} u(\lambda x, \lambda y, \lambda^3t)
$$
also solves \eqref{gzk}, with initial data
$u_\lambda(x,y,0)=\lambda^{2/k} u_0(\lambda x, \lambda y)$, for any
$\lambda>0$. Hence,
\begin{equation}\label{scaling}
\|u(\cdot,\cdot,0) \|_{\dot{H}^s}=\lambda^{2/k+s-1}
\|u_0\|_{\dot{H}^s},
\end{equation}
where $\dot{H}^s=\dot{H}^s(\mathbb{R}^2)$ denotes the homogeneous
Sobolev space of order $s$. As a consequence of \eqref{scaling}, the
scale-invariant Sobolev space for the gZK equation is
$H^{s_k}(\mathbb{R}^2)$, where $s_k=1-2/k$. Therefore, one expects
that the Sobolev spaces $H^{s}(\mathbb{R}^2)$ for studying the
well-posedness of \eqref{gzk} are those with indices $s>s_k$.

It should be noted that $s_k<3/4$ if $1\leq k\leq7$, $s_k=s_k^*=3/4$
if $k=8$, and $s_k^*>s_k$ if $k>8$. Thus, in view of Theorem
\ref{reviewtheo}, except in the case $k=8$, a gap for the local
well-posedness is left between the index conjectured by the scaling
argument and that one known in the current literature. One of our
goal here is to fulfill this gap by reaching the critical index
$s_k=1-2/k$ (up to the endpoint) in the case $k>8$. More precisely,
we prove the following.

\begin{theorem}   \label{theorem3}
Let $k>8$ and $s_k=1-2/k$. For any $u_0 \in H^s(\mathbb{R}^2)$,
$s>s_k$, there exist $T=T(\|u_0\|_{H^s})>0$ and a unique solution of
the IVP \eqref{gzk}, defined in the interval $[0,T]$, such that
\begin{equation}\label{b1.2}
u \in C([0,T];H^s(\mathbb{R}^2)),
\end{equation}
\begin{equation}\label{b2.2}
\|u_x\|_{L^\infty_xL^2_{yT}} + \|D^s_xu_x\|_{L^\infty_xL^2_{yT}}  +
\|D^s_yu_x\|_{L^\infty_xL^2_{yT}}  <\infty,
\end{equation}
\begin{equation}\label{b3.2}
 \|u\|_{L^{3k/2+}_T L^\infty_{xy}}+   \|u_x\|_{L^{3k/(k+2)}_T L^\infty_{xy}}
  <\infty,
\end{equation}
and
\begin{equation}\label{b4.2}
\|u\|_{L^{k/2}_x L^\infty_{yT}}
  <\infty.
\end{equation}
Moreover, for any $T'\in(0,T)$ there exists a neighborhood $U$ of
$u_0$ in $H^s(\mathbb{R}^2)$ such that the map
$\widetilde{u}_0\mapsto \widetilde{u}(t)$ from $U$ into the class
defined by \eqref{b1.2}--\eqref{b4.2} is smooth.
\end{theorem}

The technique to show Theorem \ref{theorem3} will be the one
developed by Kenig, Ponce, and Vega \cite{KPV}, which combines
smoothing effects, Strichartz-type estimates, and a maximal function
estimate together with the Banach contraction principle. One of the
obstacles which prevent us in proving a similar result for $k\leq7$
is that we have a maximal function estimate that holds in
$H^s(\R^2)$ only for $s>3/4$ (see Lemma \ref{lemma1}).

After proving Theorem \ref{theorem3} we turn our attention to the
issue of global well-posedness. As we already mentioned, such
question has already been addressed in \cite{Fa}, \cite{BL},
\cite{LP}, \cite{LP1}. In particular, in \cite{LP} it was proved
that if $k=2$ and $\|u_0\|_{L^2}<\|Q\|_{L^2}$ (where $Q$ is the
ground state solution) then the solution is global in $H^1$ (for
global results below $H^1(\R^2)$, see \cite{LP1}). Also, in
\cite{LP1} was showed if $k\geq3$ and $\|u_0\|_{H^1}$ is small
enough then global well posedness holds in $H^1(\R^2)$. The proof of
this last result is quite standard and relies on conservation laws
and the Gagliardo-Nirenberg inequality,
\begin{equation}  \label{gn}
\int u^{k+2}\;dxdy\leq c \|u\|_{L^2}^2\|\nabla u\|_{L^2}^k,
\end{equation}
to get an a priori estimate. Indeed, first recall that the flow of
the gZK is conserved by the quantities:
\begin{equation}\label{MC}
Mass\equiv M(u(t))=\int u^2(t)\, dxdy
\end{equation}
and
\begin{equation}\label{EC}
Energy\equiv E(u(t))=\dfrac{1}{2}\int|\nabla u(t)|^2\;dxdy
-\dfrac{1}{k+2}\int u^{k+2}(t)\;dxdy,
\end{equation}
where the symbol $\nabla$ stands for the gradient in the space
variables.

Combining \eqref{MC}, \eqref{EC} and \eqref{gn}, we obtain that
\begin{equation}\label{smallg}
\begin{split}
\| u(t)\|_{H^1}^2&= M(u(t))+ 2E(u(t))+\frac{2}{k+2}\int u^{k+2}(t)\;dxdy\\
&\leq M(u_0)+ 2E(u_0)+c\|u_0\|^2_{L^2}\|\nabla u(t)\|_{L^2}^k.
\end{split}
\end{equation}
Denote $X(t)=\|u(t)\|_{H^1}^2$. Since $k\geq3$, we then have
$$
X(t)\leq C(\|u_0\|_{H^1})+c\|u_0\|^2_{L^2}X(t)^{1+\frac{k-2}{2}}.
$$
Thus, if $\|u_0\|_{H^1}$ is small enough, a standard argument leads
to $\|u(t)\|_{H^1}\leq C(\|u_0\|_{H^1})$ for $t\in[0,T]$. Therefore,
we can apply the local theory to extend the solution.

Unfortunately, the above argument does not precise how small the
initial data should be. Here, we study this question and obtain the
following result.

\begin{theorem}\label{global1}
 Let $k\geq3$ and $s_k=1-2/k$.  Assume $u_0\in H^1(\R)$ and suppose that
\begin{equation}\label{GR1}
E(u_0)^{s_k} M(u_0)^{1-s_k} < E(Q)^{s_k} M(Q)^{1-s_k} , \,\,\,
E(u_0) \geq 0.
\end{equation}
If
\begin{equation}\label{GR2}
\|\nabla u_0\|_{L^2}^{s_k}\|u_0\|_{L^2}^{1-s_k} < \|\nabla
Q\|_{L^2}^{s_k}\|Q\|_{L^2}^{1-s_k},
\end{equation}
then for any $t$ as long as the solution exists,
\begin{equation}\label{GR3}
\|\nabla u(t)\|_{L^2}^{s_k}\|u_0\|_{L^2}^{1-s_k}=\|\nabla
u(t)\|_{L^2}^{s_k}\|u(t)\|_{L^2}^{1-s_k} <\|\nabla
Q\|_{L^2}^{s_k}\|Q\|_{L^2}^{1-s_k},
\end{equation}
where $Q$ is the unique positive radial solution of
$$
\Delta Q-Q+Q^{k+1}=0.
$$
This in turn implies that $H^1$ solutions exist globally in time.
\end{theorem}

To prove Theorem \ref{global1}, we follow closely  our arguments in
\cite{FLP} where we have proved a similar result for the
$L^2$-supercritical generalized KdV equation. We point out that
these results are inspired   by those ones obtained by Kenig and
Merle \cite{kenig-merle} and Holmer and Roudenko \cite{HR08}.

\begin{remark}\label{remark2}
In the limit case $k=2$ (the  modified ZK equation), conditions
\eqref{GR1} and \eqref{GR2} reduce to the same one and it writes as
$$
\|u_0\|_{L^2}<\|Q\|_{L^2}.
$$
Such a condition was already used in \cite{LP} and \cite{LP1} to
show the existence of global solutions, respectively, in $H^1(\R^2)$
and $H^s(\R^2)$, $s>53/63$.
\end{remark}

Once Theorem \ref{global1} is established, we go on studying the
asymptotic behavior of such global solutions as
$t\rightarrow\pm\infty$. We prove that under a smallness condition
the solution scatters to a solution of the linear problem.
Precisely,

\begin{theorem}\label{scattering}
Let $k\geq3$ and $p'=\frac{2(k+1)}{2k+1}$. Assume that $u_0\in
H^1(\R^2)\cap L^{p'}(\R^2)$ satisfies
\begin{equation}\label{smallness}
\|u_0\|_{L^{p'}}+\|u_0\|_{H^1}<\delta
\end{equation}
for some $\delta$ small enough. Let $u(t)$ be the global solution of
\eqref{gzk} given in Theorem \ref{global1}. Then, there exist
$f_{\pm}\in H^1(\R^2)$ such that
\begin{equation}\label{limit}
\|u(t)-U(t)f_{\pm}\|_{H^1}\rightarrow0,
\end{equation}
as $t\rightarrow\pm\infty$.
\end{theorem}

Note that the smallness condition \eqref{smallness} promptly implies
the existence of global solutions in $H^1(\R^2)$. The proof of
Theorem \ref{scattering} is quite standard and it follows closely
the arguments in \cite{PV}, \cite{St}.

\begin{remark} Theorem \ref{scattering} provides scattering
whenever the initial data is small in $H^1(\R^2)$ and in
$L^{p'}(\R^2)$. We do not know if the smallness condition in Theorem
\ref{global1} is sharp in the sense that any global solution given
by Theorem \ref{global1} scatters or not.
\end{remark}

The paper is organized as follows. In section \ref{preliminaries} we
introduce some notation and recall the useful linear estimates to
our arguments. The local and global results, in Theorems
\ref{theorem3} and \ref{global1}, are proved in Sections \ref{LWP}
and  \ref{GWP}, respectively. The concluding section, Section
\ref{scasec}, is devoted to show Theorem \ref{scattering}.

\section{Notation and Preliminaries}\label{preliminaries}

Let us start this section by introducing the basic  notation used
throughout this note. We use $c$ to denote various constants that
may vary line by line. Given any positive numbers $a$ and $b$, the
notation $a \lesssim b$ means that there exists a positive constant
$c$ such that $a \leq cb$. We use $a+$ and $a-$ to denote
$a+\varepsilon$ and $a-\varepsilon$, respectively, for arbitrarily
small $\varepsilon>0$.

For $\alpha \in \mathbb{C}$, the operators $D^\alpha_x$ and
$D^\alpha_y$ are defined via Fourier transform by
$\widehat{D^\alpha_x f}(\xi,
\eta)=|\xi|^\alpha\widehat{f}(\xi,\eta)$ and $\widehat{D^\alpha_y
f}(\xi, \eta)=|\eta|^\alpha\widehat{f}(\xi,\eta)$, respectively. We
use $\|\cdot\|_{L^p}$ and $\|\cdot\|_{H^s}$ to denote the norms in
$L^p(\R^2)$ and $H^s(\R^2)$, respectively. If necessary, we use
subscript to inform which variable we are concerned with. The mixed
space-time norm is defined as (for $1\leq p,q,r<\infty$)
$$
\|f\|_{L^p_xL^q_yL^r_T}= \left( \int_{-\infty}^{+\infty} \left(
\int_{-\infty}^{+\infty} \left( \int_0^T |f(x,y,t)|^r dt
\right)^{q/r} dy \right)^{p/q} dx \right)^{1/p},
$$
with obvious modifications if either $p=\infty$, $q=\infty$ or
$r=\infty$. Norms with interchanged subscript  are similarly
defined. If the subscript $L^r_t$ appears in some norm, that means
one is integrating the variable $t$ on the whole $\R$.

Next we introduce the main tools to prove the local well-posedness.
Consider the linear IVP
\begin{equation}\label{a1}
     \left\{
\begin{array}{lll}
{\displaystyle u_t+\partial_x \Delta u =  0,}  \qquad (x,y) \in \mathbb{R}^2, \,\,\,\, t \in \mathbb{R}, \\
{\displaystyle  u(x,y,0)=u_0(x,y)}.
\end{array}
\right.
\end{equation}
The solution of \eqref{a1} is given by the unitary group
$\{U(t)\}_{t=-\infty}^\infty$ such that
\begin{equation}\label{a2}
u(t)=U(t)u_0(x,y)= \int_{\mathbb{R}^2}
e^{i(t(\xi^3+\xi\eta^2)+x\xi+y\eta)}\widehat{u}_0(\xi,\eta) d\xi
d\eta.\\
\end{equation}

The Smoothing effect of Kato type, the Strichartz estimate, and the
maximal function estimate for solution \eqref{a2} are presented
next.

\begin{lemma}  \label{lemma1}
The following statements hold.
\begin{itemize}
  \item[(i)] {\bf(Smoothing effect)} If $u_0 \in L^2(\mathbb{R}^2)$ then
  \begin{equation}  \label{seffect}
    \|\partial_x U(t)u_0\|_{L^\infty_xL^2_{yT}} \lesssim
    \|u_0\|_{L^2_{xy}}.
  \end{equation}

  \item[(ii)] {\bf(Maximal function)} For any $s>3/4$ and $0< T \leq 1$, we have
\begin{equation}\label{maximal}
\|U(t)f\|_{L^4_xL^\infty_{yT}} \lesssim \|f\|_{H^s_{xy}}.
\end{equation}

  \item[(iii)] {\bf({Strichartz-type estimate})}
Let $0\leq \varepsilon <1/2$ and $0\leq \theta \leq 1$. Then,
\begin{equation}\label{a3}
    \|D^{\theta \varepsilon/2}_x U(t)f \|_{L^q_tL^p_{xy}} \lesssim
    \|f\|_{L^2_{xy}},
\end{equation}
where $p=\frac{2}{1-\theta}$ and
$\frac{2}{q}=\frac{\theta(2+\varepsilon)}{3}$.
\end{itemize}
\end{lemma}
\begin{proof}
The proof (i) is given in \cite[Theorem 2.2]{Fa} while proofs of
(ii) and (iii) are given, respectively, in \cite[Proposition
2.4]{LP} and \cite[Corollary 2.7]{LP1}.
\end{proof}

With Lemma \ref{lemma1} at hand, we prove the following.

\begin{proposition}\label{prop0}
Let $s_k=1-2/k$ and $0<T\leq1$. Then, for any $k>8$,
\begin{itemize}
\item [(i)] $\|U(t)f\|_{L^{k/2}_xL^{\infty}_{yT}}\lesssim
\|f\|_{H^{s_k+}_{xy}}$,
\item [(ii)] $\|U(t)f\|_{L^{3k/2+}_TL^{\infty}_{xy}}\lesssim
\|f\|_{H^{s_k+}_{xy}}$,
\item [(iii)] $\|\partial_xU(t)f\|_{L^{3k/(k+2)}_TL^{\infty}_{xy}}\lesssim \|D^{s_k}_xf\|_{L^{2}_{xy}}$.
\end{itemize}
\end{proposition}
\begin{proof}
Inequality (i) follows interpolating the Sobolev embedding
\begin{equation}\label{sob}
\|U(t)f\|_{L^{\infty}_{Txy}}\lesssim \|f\|_{H^{1+}}
\end{equation}
with the maximal function estimate
$\|U(t)f\|_{L^{4}_xL^{\infty}_{yT}}\leq \|f\|_{H^{s_k+}}$ (see
\eqref{maximal}). To prove (ii) we first take $\varepsilon=0$ and
$\theta=1$ in \eqref{a3} to get
$\|U(t)f\|_{L^{3}_TL^{\infty}_{xy}}\lesssim \|f\|_{L^{2}}$. Thus
(ii) follows interpolating such inequality with \eqref{sob}.
Estimate (iii) is a particular case of \eqref {a3} just taking
$\theta=1$ and $\varepsilon=4/k$.

\end{proof}

Finally, we also recall the Chain rule and Leibniz rule for
fractional derivatives.

\begin{proposition} {\bf(Chain rule)}  \label{chain}
Let $1<p<\infty$, $r>1$, and $h\in L_{\rm loc}^{rp}(\R)$. Then
\begin{equation*}
\|D^\alpha_z F(f)h\|_{L^p_z(\R)}\lesssim
\|F'(f)\|_{L_z^\infty(\R)}\|D^\alpha_z(f)M(h^{rp})^{1/{rp}}\|_{L^p_z(\R)},
\end{equation*}
where $M$ denotes the Hardy-Littlewood maximal function.
\end{proposition}
\begin{proof}
See Kenig, Ponce, and Vega \cite[Theorem A.7]{KPV}.\\
\end{proof}

\begin{lemma} {\bf (Leibniz rule)}  \label{lemmalei}
Let $0<\alpha<1$ and $1<p<\infty$. Then,
$$
\|D^\alpha_z(fg)-fD^\alpha_z g -gD^\alpha_z f \|_{L_z^p(\mathbb{R})}
\lesssim \|g\|_{L_z^\infty(\mathbb{R})} \|D_z^\alpha f
\|_{L_z^p(\mathbb{R})}.
$$
\end{lemma}
\begin{proof}
See Kenig, Ponce, and Vega \cite[Theorem A.12]{KPV}.\\
\end{proof}

\section{Local Well-posedness: Proof of Theorem \ref{theorem3}}\label{LWP}

As usual, we consider the integral operator
\begin{equation}\label{Psi}
\Psi(u)(t)=\Psi_{u_0}(u)(t):=U(t)u_0+\int_0^t
U(t-t')\partial_x(u^{k+1})(t')dt',
\end{equation}
and define the metric spaces
$$
\mathcal{Y}_T=\{ u \in C([0,T];H^s(\mathbb{R}^2)); \,\,\,\, \trinorm
u \trinorm <\infty \}
$$
and
$$
\mathcal{Y}_T^a=\{ u \in \mathcal{X}_T; \,\,\,\, \trinorm u \trinorm
\leq a \},
$$
with
\begin{equation*}
\begin{split}
\trinorm u \trinorm:= \|u\|_{L^\infty_TH^s_{xy}}+
\|u\|_{L^{3k/2+}_TL^\infty_{xy}}& + \|u\|_{L^{k/2}_xL^\infty_{yT}} +
\|u_x\|_{L^{3k/(k+2)}_TL^\infty_{xy}} \\
&+\|u_x\|_{L^\infty_xL^2_{yT}} + \|D^s_xu_x\|_{L^\infty_xL^2_{yT}} +
\|D^s_yu_x\|_{L^\infty_xL^2_{yT}},
\end{split}
\end{equation*}
where $a,T>0$ will be chosen later. We assume that $s_k<s<1$ and
$T\leq 1$.

First we estimate the $H^s$-norm of $\Psi(u)$. Let $u\in
\mathcal{Y}_T$. By using Minkowski's inequality, group properties
and then H\"older's inequality, we have
\begin{equation}\label{b5}
\begin{split}
 \|\Psi(u)(t) \|_{L^2_{xy}} & \lesssim  {\displaystyle  \|u_0\|_{H^s}+
 \int_0^T \|u^{3k/4}\|_{L^{\infty}_{xy}} \|u^{k/4}u_x\|_{L^2_{xy}} dt'  } \\
& \lesssim  \|u_0\|_{H^s}+   \|u^{3k/4}\|_{L^2_TL^{\infty}_{xy}}
 \|u^{k/4}u_x\|_{L^2_{xyT}}\\
& \lesssim \|u_0\|_{H^s}+  T^{\gamma}
\|u\|_{L^{3k/2+}_TL^{\infty}_{xy}}
\|u\|^{k/4}_{L^{k/2}_xL^\infty_{yT}}\|u_x\|_{L^{\infty}_xL^{2}_{yT}},
\end{split}
\end{equation}
where $\gamma >0$ is an arbitrarily small number.

On the other hand, using group properties and Minkowski's
inequality, we have
\begin{equation}   \label{b6}
 \|D^s_x \Psi(u)(t) \|_{L^2_{xy}}  \lesssim  {\displaystyle  \|D^s_xu_0\|_{L^2_{xy}}+\int_0^T \|D^s_x(u^ku_x)\|_{L^2_{xy}} dt'} =
 {\displaystyle c \|u_0\|_{H^s}+ A_0. }
\end{equation}

Applying Leibniz rule for fractional derivatives (see Lemma
\ref{lemmalei}) and H\"older's inequality, we get
\begin{equation}\label{b7}
\begin{split}
A_0& \lesssim \int_0^T
 \|u^{3k/4}\|_{L^\infty_{xy}} \|D_x^s(u^{k/4}u_x)\|_{L^2_{xy}}dt' +
\int_0^T
 \|D_x^s(u^{3k/4})u^{k/4}u_x)\|_{L^2_{xy}}dt'\\
&= A_1+A_2.
\end{split}
\end{equation}

Moreover,
\begin{equation}\label{b8}
\begin{split}
A_1&\lesssim \int_0^T
 \|u^{3k/4}\|_{L^\infty_{xy}}\|u_x\|_{L^\infty_{xy}}\|D_x^su^{k/4}\|_{L^2_{xy}} dt'+
\int_0^T\|u^{3k/4}\|_{L^\infty_{xy}}\|u^{k/4}D_x^su_x\|_{L^2_{xy}}dt'\\
&= A_{11}+A_{12}.
\end{split}
\end{equation}

First we consider the term $A_{11}$. Thus, applying H\"older's
inequality, Lemma \ref{chain} (with $h=1$) and H\"older's inequality
again, we have
\begin{equation}\label{b9}
\begin{split}
A_{11}  & \lesssim  \|u\|^{3k/4}_{L^{3k/2}_TL^\infty_{xy}} \|\|u_x\|_{L^\infty_{xy}} \|D_x^su^{k/4}\|_{L^2_{xy}}\|_{L^2_{T}}\\
&\lesssim  \|u\|^{3k/4}_{L^{3k/2}_TL^\infty_{xy}} \|\|u_x\|_{L^\infty_{xy}}\|u^{k/4-1}\|_{L^{\infty}_{xy}} \|D_x^su\|_{L^2_{xy}}\|_{L^2_{T}}\\
&\lesssim \|u\|^{3k/4}_{L^{3k/2}_TL^\infty_{xy}} \|u_x\|_{L^{3k/(k+2)}_TL^\infty_{xy}}\|u\|^{(k-4)/4}_{L^{3k/2}_TL^\infty_{xy}} \|u\|_{L^{\infty}_TH^s_{xy}}\\
&\lesssim T^{\gamma}\|u\|^{k-1}_{L^{3k/2+}_TL^\infty_{xy}} \|u_x\|_{L^{3k/(k+2)}_TL^\infty_{xy}}\|u\|_{L^{\infty}_TH^s_{xy}}.\\
\end{split}
\end{equation}

To bound $A_{12}$ we just apply H\"older's inequality twice to obtain
\begin{equation}\label{b10}
\begin{split}
A_{12}  & \lesssim  \|u\|^{3k/4}_{L^{3k/2}_TL^\infty_{xy}} \|D^s_xu_x\|_{L^\infty_{x}L^{2}_{yT}} \|u\|^{k/4}_{L^{k/2}_{x}L^{\infty}_{yT}}\\
& \lesssim  T^{\gamma}\|u\|^{3k/4}_{L^{3k/2+}_TL^\infty_{xy}} \|D^s_xu_x\|_{L^\infty_{x}L^{2}_{yT}} \|u\|^{k/4}_{L^{k/2}_{x}L^{\infty}_{yT}}.\\
\end{split}
\end{equation}

Next we consider the term $A_2$. Lemma \ref{chain} (with $h=1$) and
H\"older's inequality yield
\begin{equation}\label{b11}
\begin{split}
A_{2}  & \lesssim \int_0^T
 \|D_x^s(u^{3k/4})\|_{L^2_{xy}}\|u^{k/4}u_x\|_{L^{\infty}_{xy}}dt'\\
& \lesssim \int_0^T
 \|u^{3k/4-1}\|_{L^{\infty}_{xy}} \|D_x^su\|_{L^2_{xy}}\|u\|_{L^{\infty}_{xy}}^{k/4}\|u_x\|_{L^{\infty}_{xy}}dt'\\
 & \lesssim  \|u\|^{k-1}_{L^{3k/2}_TL^\infty_{xy}} \|u_x\|_{L^{3k/(k+2)}_TL^\infty_{xy}}\|u\|_{L^{\infty}_TH^s_{xy}}\\
& \lesssim  T^{\gamma}\|u\|^{k-1}_{L^{3k/2+}_TL^\infty_{xy}} \|u_x\|_{L^{3k/(k+2)}_TL^\infty_{xy}}\|u\|_{L^{\infty}_TH^s_{xy}}.\\
\end{split}
\end{equation}

A similar analysis can be carried out to estimate the norm $\|D^s_y \Psi(u)(t) \|_{L^2_{xy}}$. Therefore, from \eqref{b5}-\eqref{b11}, we deduce
\begin{equation}\label{b12}
    \|\Psi(u)\|_{L^\infty_TH^s} \leq  c \|u_0\|_{H^s} +  cT^{\gamma} \trinorm u
    \trinorm^{k+1}.
\end{equation}

The remaining norms are estimated similarly. Indeed, by combining
the linear estimates (i)-(iii) in Proposition \ref{prop0}, Lemma
\ref{lemma1} (i), and group properties it is easy to see that all
the problem reduces to the estimation of $A_0$. Therefore, we infer
$$
\trinorm \Psi(u)\trinorm \leq  c\|u_0\|_{H^s} + cT^{\gamma}
\trinorm u
    \trinorm^{k+1}.
$$
Choose $a=2c \|u_0\|_{H^s}$, and $T>0$ such that
$$
c\,a^k T^{\gamma} \leq \frac{1}{4}.
$$
Then, we see that $\Psi:\mathcal{Y}_T^a \mapsto \mathcal{Y}_T^a$ is
well defined. Moreover, similar arguments show that $\Psi$ is a
contraction. To finish the proof we use standard arguments, thus, we
omit the details. This completes the proof of Theorem
\ref{theorem3}.

\section{Global Well-posedness: Proof of Theorem \ref{global1}}\label{GWP}

We first note that from the discussion in \eqref{smallg} the
smallness condition on $\|u_0\|_{H^1}$ should be closely related to
the constant appearing in  the Gagliardo-Nirenberg inequality
\eqref{gn}. Thus, let us recall the classical result obtained by
Weinstein \cite{W83}, regarding the best constant for the
Gagliardo-Nirenberg inequality.

\begin{theorem}\label{best}
Let $k>0$,  then the Gagliardo-Nirenberg inequality
\begin{equation}\label{g-n}
\|u\|_{L^{k+2}}^{k+2}\le K_{\rm opt}^{k+2}\,\|\nabla u\|_{L^2}^{k}\|u\|_{L^2}^{2},
\end{equation}
holds, and the sharp constant $K_{\rm opt}>0$ is
explicitly given by
\begin{equation}\label{opt}
K_{\rm opt}^{k+2}=\frac{k+2}{2\|\psi\|_{L^2}^k},
\end{equation}
where $\psi$ is the unique non-negative, radially-symmetric, decreasing solution of the equation
\begin{equation}\label{ground}
\frac{k}{2}\Delta \psi-\psi+\psi^{k+1}=0.
\end{equation}
\end{theorem}
\begin{proof}
See \cite[Corollary 2.1]{W83}.
\end{proof}

\begin{remark}\label{remark1}
If $\psi$ is the solution of \eqref{ground}, then by uniqueness
$$
Q(x,y)=\psi\left(\sqrt{\frac k 2}(x,y)\right),
$$
is the solution of
\begin{equation}\label{ground3}
\Delta Q-Q+Q^{k+1}=0.
\end{equation}
Moreover,
$$
\|Q\|^2_{L^2}=\frac{2}{k} \|\psi\|^2_{L^2}.
$$
\end{remark}

In view of Remark \ref{remark1} and \eqref{opt}, we deduce that
\begin{equation}\label{opt1}
K_{\rm opt}^{k+2}=\frac{2^{\frac{k-2}{2}}(k+2)}{k^{\frac k 2}\|Q\|_{L^2}^k}.
\end{equation}
Now, by multiplying (\ref{ground3}) by $Q$, integrating, and
applying integration by parts, we obtain
$$
\int_{\R^2}Q^{k+2}\, dxdy=\|Q\|_{L^{2}}^{2}+\|\nabla Q\|_{L^{2}}^{2}.
$$
On the other hand, by multiplying (\ref{ground3}) by $(x,y)\cdot
\nabla Q$, integrating, and applying integration by parts, we
promptly  obtain the  identity
$$
\int_{\R^2}Q^{k+2}\, dxdy=\frac{k+2}{2}\|Q\|_{L^{2}}^{2}.
$$
Combining the last two relations, we have
\begin{equation}\label{relQ}
\frac{ k}{ 2} \|Q\|_{L^{2}}^{2}=\|\nabla Q\|_{L^{2}}^{2}.
\end{equation}
With these tools at hand, we are able to prove Theorem
\ref{global1}.

\begin{proof}[Proof of Theorem \ref{global1}]
We proceed as follows: write the $\dot{H}^1$-norm of $u(t)$ using
the quantities $M(u(t))$ and $E(u(t))$. Then we use the sharp
Gagliardo-Nirenberg inequality \eqref{g-n}, with the sharp constant
in \eqref{opt1}, to yield
\begin{equation}\label{ap10}
\begin{split}
\|\nabla u(t)\|_{L^2}^2&=2E(u_0)+\frac{2}{k+2} \int_{\R^2}u^{k+2}(t)\,dxdy\\
&\le 2E(u_0)+\frac{2}{k+2}K_{\rm opt}^{k+2}\,
\|u_0\|_{L^2}^{2}\|\nabla u(t)\|_{L^2}^{k}\\
&=2E(u_0)+\left(\frac{2}{k}\right)^{\frac{k}{2}}\frac{1}{\|Q\|_{L^2}^k}\|u_0\|_{L^2}^2\|\nabla u(t)\|_{L^2}^k.
\end{split}
\end{equation}
Let $X(t)=\|\nabla u(t)\|_{L^2}^2$, $A=2E(u_0)$, and
$B=\left(\frac{2}{k}\right)^{\frac{k}{2}}\frac{\|u_0\|_{L^2}^2}{\|Q\|_{L^2}^k}$,
then we can write \eqref{ap10} as
\begin{equation}\label{ap12}
X(t)-B\,X(t)^{k/2}\le A, \text{\hskip2pt for}\;\;t\in (0,T),
\end{equation}
where $T$ is given by Theorem \ref{reviewtheo} (or Theorem
\ref{theorem3} if $k>8$).

Now let $f(x)=x-B\,x^{k/2}$, for $x\ge 0$. The function $f$ has a
local maximum at $x_0=\Big(\dfrac{2}{kB}\Big)^{2/(k-2)}$ with
maximum value
$f(x_0)=\dfrac{k-2}{k}\Big(\dfrac{2}{kB}\Big)^{2/(k-2)}.$ If we
require that
\begin{equation}\label{ap13}
E(u_0) < f(x_0)\,\,\,\, \mbox{and}\, \,\,\, X(0) < x_0,
\end{equation}
the continuity of $X(t)$ implies that $X(t) < x_0$ for any $t$ as long as the solution exists.

Using relations (\ref{relQ}), we have
$$
E(Q)=\dfrac{k-2}{4}\|Q\|_{L^2}^2.
$$
Therefore, a simple calculation shows that conditions (\ref{ap13})
are exactly the inequalities (\ref{GR1}) and (\ref{GR2}). Moreover
the inequality $X(t) < x_0$ reduces to (\ref{GR3}). The proof of
Theorem \ref{global1} is thus completed.
\end{proof}

\section{Scattering: Proof of Theorem \ref{scattering}}\label{scasec}

We start by recalling the following decay result for solutions
$u(t)=U(t)f$, of the linear problem \eqref{a1}.

\begin{proposition}\label{prop1}
Let $0\leq \varepsilon<1/2$ and $0\leq \theta \leq1$. Then,
$$
\|D_x^{\theta\varepsilon}U(t)f\|_{L^p_{xy}}\leq C|t|^{-\theta\frac{(2+\varepsilon)}{3}}\|f\|_{L^{p'}_{xy}},
$$
where $p=\frac{2}{1-\theta}$ and $p'=\frac{2}{1+\theta}$. In
particular,
$$
\|U(t)f\|_{L^p_{xy}}\leq C|t|^{-\frac{2\theta}{3}}\|f\|_{L^{p'}_{xy}}.
$$
\end{proposition}
\begin{proof}
See Linares and Pastor \cite[Lemma 2.3]{LP}.
\end{proof}

As a consequence, we have.

\begin{corollary}\label{cor1}
Let $p$ and $p'$ be as in Proposition \ref{prop1}. If $0\leq \theta<1$ and $f\in L^{p'}(\R^2)\cap H^1(\R^2)$, then
$$
\|U(t)f\|_{L^p_{xy}}\leq C(1+|t|)^{-\frac{2\theta}{3}}(\|f\|_{L^{p'}}+\|f\|_{H^1}).
$$
\begin{proof}
The proof follows immediately from Proposition \ref{prop1} and the
embedding of $H^1$ in $L^p$, $2\leq p<\infty$.
\end{proof}
\end{corollary}

\begin{theorem}[Decay] \label{decaytheorem}
Let $p=2(k+1)$, $p'=\frac{2(k+1)}{2k+1}$, and
$\theta=\frac{k}{k+1}$. Assume  $u_0\in L^{p'}(\R^2)\cap H^1(\R^2)$
satisfies
$$
\|u_0\|_{L^{p'}}+\|u_0\|_{H^1}<\delta.
$$
Then, the solution $u(t)$ given in Theorem \ref{global1} satisfies
$$
(1+|t|)^{\frac{2\theta}{3}}\|u(t)\|_{L^p}\leq C
$$
for all $t\in\R$ and some constant $C>0$.
\end{theorem}
\begin{proof}
From the integral formulation of \eqref{gzk}, we have
$$
u(t)=U(t)u_0-\int_0^tU(t-t')\p_x(u^{k+1})(t')\,dt'.
$$
Without loss of generality assume $t>0$. Thus, from Proposition
\ref{prop1} and Corollary \ref{cor1}, we have
\begin{equation*}
\begin{split}
\|u(t)\|_{L^p}&\leq \|U(t)u_0\|_{L^{p}}+\int_0^t\|U(t-t')\p_x(u^{k+1})(t')\|_{L^p}\,dt'\\
& \leq C(1+t)^{-\frac{2\theta}{3}}(\|u_0\|_{L^{p'}}+\|u_0\|_{H^1})+C\int_0^t(t-t')^{-\frac{2\theta}{3}}
\|\p_x(u^{k+1})(t')\|_{L^{p'}}\,dt'\\
&\leq C(1+t)^{-\frac{2\theta}{3}}\delta+C\int_0^t(t-t')^{-\frac{2\theta}{3}}
\|u^k\p_xu(t')\|_{L^{p'}}\,dt'\\
&\leq C(1+t)^{-\frac{2\theta}{3}}\delta+C\int_0^t(t-t')^{-\frac{2\theta}{3}} \|u^k\|_{L^{\frac{2(k+1)}{k}}}\|\p_xu\|_{L^2}\,dt'\\
&\leq C(1+t)^{-\frac{2\theta}{3}}\delta+C\|u\|_{L^\infty_TH^1}\int_0^t(t-t')^{-\frac{2\theta}{3}} \|u(t')\|_{L^p}^k\,dt'.
\end{split}
\end{equation*}
Let
$$
M(T)=\sup_{t\in[0,T]}(1+t)^{\frac{2\theta}{3}}\|u(t)\|_{L^p}.
$$
Then,  we can write
\begin{equation} \label{M(T)}
M(T)\leq C\delta+C\delta(1+t)^{\frac{2\theta}{3}}M(T)^k\int_0^t(t-t')^{-\frac{2\theta}{3}} (1+t')^{\frac{2\theta k}{3}}\,dt'.
\end{equation}
Since $k\geq3$, we then obtain
$$
M(T)\leq C\delta+C\delta M(T)^k.
$$
Hence, if $\delta\ll1$, we deduce from a continuity argument that
$M(T)\leq C$. This completes the proof.
\end{proof}

\begin{remark}
From \eqref{M(T)} we see that it suffices to take
$k>\frac{3+\sqrt{33}}{4}\simeq 2.186$. Note that the case $k=2$
($L^2$-critical) is not cover by our result and it is a very
interesting open problem.
\end{remark}

In the proof of Theorem \ref{scattering}, we only consider the case
as $t\rightarrow-\infty$, since that as $t\rightarrow+\infty$ is
similarly treated. Define
$$
f_{-}=u_0-\int_{-\infty}^0U(-t')\pa_x(u^{k+1})\,dt'.
$$
Then,
$$
u(t)-U(t)f_{-}=\int_{-\infty}^tU(t-t')\pa_x(u^{k+1})\,dt'.
$$

\begin{lemma} \label{lemmat}
$\|U(-t)u(t)-f_{-}\|_{L^{2(k+1)}}\rightarrow0$, as
$t\rightarrow-\infty$.
\end{lemma}
\begin{proof}
Indeed, from Proposition \ref{prop1}, we
$$
\|U(-t)u(t)-f_{-}\|_{L^{2(k+1)}}\leq
C\int_{-\infty}^t|t'|^{-\frac{2k}{3(k+1)}}\|u^k\pa_xu\|_{L^{\frac{2(k+1)}{2k+1}}}\,dt'.
$$
From H\"older's inequality and Theorem \ref{global1}, we then deduce
\begin{equation*}
\begin{split}
\|U(-t)u(t)-f_{-}\|_{L^{2(k+1)}}\leq &
C\int_{-\infty}^t|t'|^{-\frac{2k}{3(k+1)}}
\|u^k\|_{L^{\frac{2(k+1)}{k}}}\|\pa_xu\|_{L^2}\,dt'\\
\leq& C\|u\|_{L_t^\infty H^1}C \int_{-\infty}^t|t'|^{-\frac{2k}{3(k+1)}}\|u\|_{L^{2(k+1)}}^k\,dt'\\
\leq& C\left( \int_{-\infty}^t|t'|^{\frac{4k}{3}}
\right)^{\frac{1}{2(k+1)}} \left(
\int_{-\infty}^t\|u\|_{L^{2(k+1)}}^{\frac{2k(k+1)}{2k+1}}\,dt'
\right)^{\frac{2k+1}{2(k+1)}}.
\end{split}
\end{equation*}
Since $k\geq3$ these last two integrals tend to zero as
$t\rightarrow-\infty$.
\end{proof}

\begin{lemma}\label{lemmaG}
Let
$$
G(u)=\frac{1}{k+2}\int_{\R^2}u^{k+2}\,dx.
$$
Then, $G(u(t))\rightarrow0$, as $t\rightarrow-\infty$.
\end{lemma}
\begin{proof}
From H\"older's inequality and Theorem \ref{decaytheorem}, we have
\begin{equation*}
\begin{split}
|G(u(t))|&\leq C\int_{\R^2}|u(t)|^{k+1}|u(t)|\,dx\\
& \leq C \|u(t)\|_{L^2}\left(\int_{\R^2}|u(t)|^{2(k+1)}\,dx\right)^{\frac{1}{2}}\\
&\leq C\|u(t)\|_{L^{2(k+1)}}^{k+1}\leq C (1+|t|)^{-\frac{2k}{3}}.
\end{split}
\end{equation*}
\end{proof}

\begin{proof}[Proof of Theorem \ref{scattering}]
Since $U(t)$ is a unitary group, from  Theorem \ref{global1}, we
obtain
$$
\|U(-t)u(t)\|_{H^1}=\|u(t)\|_{H^1}\leq C\|u_0\|_{H^1}.
$$
Thus $U(-t)u(t)\rightharpoonup f_{-}$ in $H^1$, as
$t\rightarrow-\infty$. Moreover,
\begin{equation*}
\begin{split}
\|f_{-}\|_{H^1}&\leq \liminf_{t\rightarrow-\infty}\|U(-t)u(t)\|_{H^1}=\liminf_{t\rightarrow-\infty}\|u(t)\|_{H^1}\\
&=\liminf_{t\rightarrow-\infty}\left(\|u(t)\|_{H^1}-2G(u(t))\right)\leq \|f_{-}\|_{H^1}.
\end{split}
\end{equation*}
Hence, the weak limit is strong and we have
$$
\|u(t)-U(t)f_{-}\|_{H^1}=\|U(-t)u(t)-f_{-}\|_{H^1}\rightarrow0,
$$
as $t\rightarrow-\infty$. This completes the proof of Theorem
\ref{scattering}.
\end{proof}

\section*{Acknowledgment}

\noindent L.G. Farah is partially supported by CNPq and
FAPEMIG/Brazil, F. Linares is partially supported by CNPq and
FAPERJ/Brazil, and A. Pastor is partially supported by CNPq and
FAPESP/Brazil.

\bibliographystyle{mrl}

\end{document}